\documentclass[11pt, reqno]{amsart}
\usepackage{amssymb,amsrefs, amsmath, amsthm, color, marginnote,hyperref,cancel,amscd, amsfonts,mathtools}
\usepackage{enumitem}

\newtheorem{theorem}{Theorem}
\newtheorem{lemma}[theorem]{Lemma}
\newtheorem{corollary}[theorem]{Corollary}
\newtheorem{proposition}[theorem]{Proposition}
\newtheorem*{question}{Questions}

\theoremstyle{definition}
\newtheorem{definition}[theorem]{Definition}

\theoremstyle{remark}
\newtheorem{remark}[theorem]{Remark}

\numberwithin{equation}{section}

\newcommand{\HZ}{\mathcal{HZ}}
\newcommand{\ku}{{\Bbbk}} 

\newcommand{\pf}{\begin{proof}}
\newcommand{\epf}{\end{proof}}

\begin{document}

\title[On the bosonization of  a Lie superalgebra]{On the bosonization of the enveloping algebra of a finite dimensional Lie superalgebra}


\author{Nicolás Andruskiewitsch}
\address{Facultad de Matem\'atica, Astronom\'ia y F\'isica,
Universidad Nacional de C\'ordoba. CIEM -- CONICET. 
Medina Allende s/n (5000) Ciudad Universitaria, C\'ordoba, Argentina}
\curraddr{}
\email{nicolas.andruskiewitsch@unc.edu.ar}
\thanks{The work of N.A. was  partially supported by CONICET (PIP 11220200102916CO),FONCyT-ANPCyT (PICT-2019-03660), by the Secyt (UNC)
and by the International Center of Mathematics, Southern University
of Science and Technology, Shenzhen.}

\author{Ken A. Brown}
\address{School of Mathematics and Statistics\\
University of Glasgow\\ Glasgow G12 8QW\\
Scotland}
\email{ken.brown@glasgow.ac.uk}

\thanks{The work of K.B. was supported by Leverhulme Emeritus Fellowship EM 2017-081.}


\date{}

\begin{abstract}
We exhibit a PI Hopf algebra that is not a finite module over its center.
We survey some ring-theoretical properties of the bosonizations of 
enveloping algebras of Lie superalgebras.  
\end{abstract}

\maketitle



\section{Introduction}
It is well-known that a ring which is a finite module over its center satisfies a polynomial identity
(it is PI, for short), cf. \S\,\ref{PI}. But the converse is false, even for  affine prime Noetherian rings, see \S \ref{positive}.
Nevertheless, one might hope that the converse is true when one restricts attention to algebras which are not ``too pathological'' - for example a prime Noetherian PI-algebra which is a maximal order is always finite over its center, see \cite[Propositions 13.9.8, 13.9.11]{Mcconell-Robson}. 

\medbreak
There are various sorts of quantum groups that are PI because they are finite over their centers. Thus it is natural to consider the following questions, see \cites{Brown-survey-PI,Goodearl-survey}:

\begin{question} Let $H$ be a PI Hopf algebra. Under which conditions is $H$ 
a finite module either    

\begin{enumerate}[leftmargin=*,label=\rm{(\alph*)}]
\item\label{item:question-center}  over its center $\mathcal{Z}(H)$, or equivalently over a central subalgebra? 

\item\label{item:question-hopf-center} Or over its Hopf center $\HZ(H)$, 
or equivalently over a central Hopf subalgebra? 
\item\label{item:question-hopf-normal} Or at least over a normal commutative one?
\end{enumerate} 
\end{question}
Here the Hopf center is the largest central Hopf subalgebra. 
It was shown in \cite{Gelaki-Letzter}
 that  being semiprime Noetherian is not enough to answer
 question \ref{item:question-hopf-normal} positively:

\begin{theorem} \cite{Gelaki-Letzter}
The bosonization $H(\mathfrak{gl}(1,1))$ of the enveloping algebra of the Lie superalgebra
$\mathfrak{gl}(1,1)$ is a semiprime Noetherian PI Hopf algebra,
 but is not a finite module over any normal commutative Hopf subalgebra.  
\end{theorem}
Notice that the main result in \cite{Gelaki-Letzter} claims that $H(\mathfrak{gl}(1,1)$ is prime rather than semiprime,
but this is not the case as we show in Lemma \ref{lema:GL-gap}.
Our main result answers the question \ref{item:question-center} negatively for  the
same Hopf algebra.

\begin{theorem} \label{thm:main}
The bosonization $H(\mathfrak{gl}(1,1))$ is not a finite module over its center.  
\end{theorem}

It is open whether such an example still exists if `semiprime Noetherian' is strengthened to
`prime Noetherian' or even `Noetherian domain`.

\medbreak
Clearly \ref{item:question-hopf-center} implies \ref{item:question-center}, 
so we wonder whether the converse is true. 
This is indeed the case for enveloping algebras in positive characteristic by a classical
Theorem of Jacobson. Also, this was asked in \cite[Question 3.6]{Brown-Zhang-jncg} 
in the setting of Hopf algebras in positive characteristic that are iterated Ore extensions. 

\medbreak
The paper is organized as follows: In Section \ref{sect:ring-th} we collect 
basic facts on bosonizations of enveloping algebras of Lie superalgebras.
Section \ref{section:centers} contains the proof of Theorem 
\ref{thm:main}. Finiteness of semiprime Noetherian PI Hopf algebras over their centers
is discussed in Section \ref{positive}.
Our reference for Hopf algebras is \cite{radford-book}; recent surveys on infinite-dimensional Hopf algebras are \cites{andrus,Brown-Zhang-survey}.

\subsection{Notation and definitions}\label{one} 
Let $\ku $ be an algebraically closed field of characteristic 0. 
All vector spaces, algebras, Lie algebras and so on are over $\ku$. 

Let $\mathfrak{g} = \mathfrak{g}_0 \oplus \mathfrak{g}_1$ 
be a  finite dimensional Lie superalgebra with even part
 $\mathfrak{g}_0$, odd part $\mathfrak{g}_1$ and  superbracket
 $[\, , \,]$. 
Denote the enveloping algebra
 of $\mathfrak{g}$ by $U(\mathfrak{g})$ and  write $H(\mathfrak{g})$
 for the bosonization of $U(\mathfrak{g})$ with the group algebra of the cyclic
 group $C_2$. 
 Thus $H(\mathfrak{g})$ is a smash product $U(\mathfrak{g}) \# \ku C_2$ 
 where the generator $t$ of $C_2$ acts on $\mathfrak{g}_0$ 
 with eigenvalue 1 and on $\mathfrak{g}_1$ with eigenvalue $-1$;
  $H(\mathfrak{g})$ is a pointed Hopf algebra where the elements of $\mathfrak{g}_0$ 
 are primitive, those  of $\mathfrak{g}_1$ are $(t,1)$-skew primitive 
 and $t$ is grouplike. Below $\vert x \vert$  denotes the degree of a homogeneous 
 element $x$ of a $C_2$-graded vector space.

\medbreak

\subsection{Key example}\label{subsec:key} 

Let $\mathfrak{g} = \mathfrak{gl}(1,1)$, that is, $\mathfrak{g} = M_2(\ku)$, decomposed so that
$ \mathfrak{g}_0  = \ku x\oplus \ku y$, $\mathfrak{g}_1  =  \ku u \oplus \ku v$,
with the superbracket given by $\mathfrak{g}_0$  being abelian and
\begin{align*}
[u,v] &= x, & [u,u] = [v,v] &=0, & [x,u] &= [x,v]  = 0, & [y,u] &= u, & [y,v] & =  -v.
\end{align*}

Hence $U(\mathfrak{gl}(1,1))$ is presented by generators
$x,y,u,v$ with defining relations
\begin{align} \label{eq:def-rels-gl1,1}
	\begin{aligned}
		&  x\text{ is central},& yu-uy &= u, & yv-vy &= -v,
		\\
		&u^2 = 0, &v^2 &=0,   &  uv + vu &= x.
	\end{aligned}
\end{align}

Then  $ H(\mathfrak{gl}(1,1))  =  U(\mathfrak{gl}(1,1)) \# \ku  C_2$
is presented by generators $x,y,u,v,t$ with  relations \eqref{eq:def-rels-gl1,1} and
\begin{align*}
t^2  &= 1, &  tx &= xt, & ty &= yt, &  tu &= -ut, &  tv &= -vt.
\end{align*}

\medbreak
In general,  given $m, n \in \mathbb N$, 
$\mathfrak{gl}(m,n) \coloneqq M_{m+n}(\ku)$, 
with elements viewed as block matrices
$X = \begin{pmatrix}A & B \\ C& D\end{pmatrix}$, 
where $A \in M_m(\ku)$ and $D \in M_n(\ku)$. 
Here 
$\mathfrak{g}_0  = \left\{\begin{pmatrix}A & 0 \\ 0& D\end{pmatrix}\right\}$, 
while  
$\mathfrak{g}_1  = \left\{\begin{pmatrix}0 & B \\ C& 0\end{pmatrix}\right\}$.
Thus in the above presentation of $\mathfrak{gl}(1,1)$ the generators are
$x = \operatorname{id} = e_{11} + e_{22}$, $y = e_{11}$, $u = e_{12}$ and $v = e_{21}$. 

\medbreak
For $X \in \mathfrak{gl}(m,n)(\ku)$ as above, its \emph{supertrace} is $\mathrm{str}(X) \coloneqq  \mathrm{tr}(A) - \mathrm{tr}(D)$.

\section{Ring-theoretical properties}\label{sect:ring-th} 
We review some properties satisfied by $H(\mathfrak{gl}(1,1))$, as consequences of more general statements. In this Section,
$\mathfrak{g}$ denotes a finite dimensional Lie superalgebra.

\subsection{Noetherianity and Gelfand-Kirillov dimension}\label{max} 
By definition, $U(\mathfrak{g})$ is affine (i.e. finitely generated as algebra). 
Now it follows from the PBW theorem for enveloping algebras of Lie superalgebras
\cite[6.1.1]{Musson} that $U(\mathfrak{g})$ is Noetherian; in particular, by
\cite[6.1.3]{Musson} $U(\mathfrak{g})$ is a finitely generated module over its
subalgebra $U(\mathfrak{g}_0)$, which is Noetherian by the classical PBW theorem.

\medbreak
Similarly,  $H(\mathfrak{g})$ is a finite right (and left) module over $U(\mathfrak{g})$, 
and so it is an affine Noetherian pointed Hopf algebra.

\medbreak
By \cite[Prop. 5.5]{Krause-Lenagan}, if $A$ is a subalgebra of an algebra $B$ 
and $B$ is a finitely generated right or left $A$-module, then 
$\operatorname{GKdim}A = \operatorname{GKdim}B$. By \cite[Ex. 6.9]{Krause-Lenagan},
\begin{align*}
\operatorname{GKdim} H(\mathfrak{g})  = 
\operatorname{GKdim} U(\mathfrak{g}) =  
\operatorname{GKdim} U(\mathfrak{g}_0)  =  
\operatorname{dim}\mathfrak{g}_0. 
\end{align*}
\medbreak

\subsection{Primeness}\label{prime} 
We start with a general fact.
If $y_1, \dots, y_m$  is a basis of $\mathfrak{g}_1$, set 
\begin{align*}
D(\mathfrak{g}) \coloneqq \det ([y_i, y_j])_{1 \le i,j \le m}\in S(\mathfrak{g}_0),
\end{align*}
which does not depend on the basis up to a nonzero scalar multiple. 

\begin{theorem}\label{primethm1}  \cite{Bell}, \cite[15.4.1]{Musson}.
If $D(\mathfrak{g}) \neq 0$, then $U(\mathfrak{g})$ is prime.
\end{theorem}

\noindent\emph{Sketch of the proof.} 
Consider the  Clifford filtration of $U(\mathfrak{g})$ 
where the elements of $\mathfrak{g}_1$ have degree 1, 
and those of $\mathfrak{g}_0$ have degree 2. 
In the associated graded algebra $R$, 
$U(\mathfrak{g}_0)$ is replaced by the symmetric algebra $S(\mathfrak{g}_0)$, which is central in $R$. Then $R$ is prime provided that
$ D(\mathfrak{g})  \neq  0$. 
By a standard argument,   so is $U(\mathfrak{g})$. 

\medbreak
It is apparently not known whether primeness always fails when $D(\mathfrak g) = 0$.

\medbreak
Recall that $\mathfrak{g}$ is  \emph{classical simple} if it is simple and $\mathfrak{g}_0$ acts semisimply on $\mathfrak{g}_1$.  Below, $\mathfrak{sl}(m,n)$ denotes the subalgebra of $\mathfrak{gl}(m,n)$ with supertrace 0; for the definition of $\mathfrak{d}(n)$, $(n \geq 2)$, see \cite[$\S$3.5]{Bell}.

\begin{theorem}\label{primethm} \cite[Theorem 3.6]{Bell}, \cite[15.4.2]{Musson}.  
Suppose that $\mathfrak{g}$ is a direct sum of Lie superalgebras each of which either
\begin{enumerate}[leftmargin=*,label=\rm{(\roman*)}]
\item has zero odd part, or 

\item  is isomorphic to some $\mathfrak{gl}(m, n)$ or $\mathfrak{sl}(m, n)$ or $\mathfrak{d}(n)$, or

\item is a classical simple Lie superalgebra which is not of type $\mathfrak{p}(n)$.
\end{enumerate} 
Then $D(\mathfrak{g}) \neq 0$ and so $U(\mathfrak{g})$ is prime. 
\end{theorem}

The enveloping algebras $U(\mathfrak{p}(n))$ 
are \emph{not} prime, but do have a unique minimal prime ideal, see \cite{Kirkman-Kuzmanovich}. 
If $\mathfrak{g}$ is a classical simple Lie superalgebra, then $U(\mathfrak{g})$ is a domain 
iff $\mathfrak{g} \simeq \mathfrak{osp}(1, 2n)$ \cite[p. 349]{Musson}.

\medbreak
However the  passage  from $U(\mathfrak{g})$ to $H(\mathfrak{g})$ is not guaranteed.
We observe that it is claimed in \cite[3.8]{Gelaki-Letzter}
that $H(\mathfrak{gl}(1,1))$ is prime, invoking \cite[2.3iii]{Goodearl-Letzter}.
But this last result cannot be applied if the action of the grouplike is inner. In actual fact we have:

\begin{lemma}\label{lema:GL-gap} 
$H(\mathfrak{gl}(1,1))$ is semiprime but not prime.
\end{lemma}

\begin{proof}
First note that $H(\mathfrak{gl}(1,1))$ is semiprime by \cite[Theorem 7]{Fisher-Montgomery}, since it is a skew group algebra $U(\mathfrak{gl}(1,1))\ast \langle t \rangle$, where $|t| = 2$ and $U(\mathfrak{gl}(1,1))$ has no 2-torsion.
 
To see that primeness fails, set $w \coloneqq -x + 2uv$. Then the following hold:
\begin{align}\label{eq:GL-gap1}
&\begin{aligned}
wx &= xw, & wy &= yw, & wu &= -uw, & wv &= -vw,& wt &= tw,
\end{aligned}
\\\label{eq:GL-gap2}
& wt \in \mathcal{Z} \big(H(\mathfrak{gl}(1,1))\big),
\\\label{eq:GL-gap3}
& (x-wt)(x + wt) =  x^2 - w^2 = 0.
\end{align}

We prove  \eqref{eq:GL-gap1}: The first equality is evident since $x$ is central.
Now
\begin{align*}
wy &= (-x + 2uv)y = -xy + 2u(yv +v) = -yx + 2(yu - u)v+ 2uv = yw;
\\
wu &= (-x + 2uv)u = -xu + 2u(x - uv) = ux;
\\
uw &= u(-x + 2uv) = -ux = -wu;
\\
wv &= (-x + 2uv)v = -xv;
\\
vw &= v (-x + 2uv) = -xv + 2(x - uv)v = xv = -wv.
\end{align*}
The  proof of \eqref{eq:GL-gap2} and the first equality in \eqref{eq:GL-gap3}
follow at once. Finally,
\begin{align*}
w^2 &= (-x + 2uv)(-x + 2uv) = x^2 -4xuv  + 4 uvuv 
\\& =  x^2 -4xuv  + 4 (x - vu)uv = x^2.
\end{align*}
Set $H = H(\mathfrak{gl}(1,1))$ and consider the left ideals $I = H(x-wt)$,
$J = H(x + wt)$. By \eqref{eq:GL-gap2}, $(x-wt)$ and $(x + wt)$ are central, hence
$I$ and $J$ are ideals. Since $IJ = 0$ by \eqref{eq:GL-gap3}, $H$ is not prime.
\end{proof} 

\begin{remark}
The ideals $I$ and $J$ in the proof can in fact be shown to be prime, and are the only minimal prime ideals of $H = H(\mathfrak{gl}(1,1))$. Indeed, Since $H$ is a skew group algebra 
$U(\mathfrak{gl}(1,1))\ast \langle t \rangle$, where $|t| = 2$, 
there are at most 2 minimal primes by \cite{Lorenz-Passman}.

\end{remark}

\subsection{Polynomial identity}\label{PI} 
First,  if a ring $S$ is a finite module over a subring $R$ which is PI, then $S$ is also PI, 
see e.g. \cite[Corollary 13.4.9]{Mcconell-Robson} (and vice versa, as a subring of  a PI ring is evidently PI). Since commutative algebras are PI, algebras that are finite modules
over the center are PI.
Second, the enveloping algebra of a Lie algebra $\mathfrak{p}$
satisfies a PI if and only if $\mathfrak{p}$ is abelian \cite{Latyshev} (recall that $\operatorname{char} \ku =0$). 
Indeed we have the following theorem. 

\begin{theorem}\label{PIthm} \cite{Bakhturin} 
The enveloping algebra $U(\mathfrak{g})$  is PI if and only if $\mathfrak{g}_0$ is abelian.
\end{theorem}

\begin{corollary}\label{PIcor} 
$H(\mathfrak{g})$ is PI if and only if $\mathfrak{g}_0$ is abelian.
\end{corollary}

This reduces drastically the chances that $H(\mathfrak{g})$ is prime and 
PI simultaneously.

\section{Centers}\label{section:centers}

\subsection{Center of $U(\mathfrak{g})$}\label{centerU} 
We shall use the following result.

\begin{lemma}\label{lema:Noetherian center} \cite[13.6.14]{Mcconell-Robson}
Let $R$ be a semiprime PI ring. The  following are equivalent:
\begin{itemize}
\item the center $\mathcal{Z}(R)$ is Noetherian;

\item $R$ is right Noetherian and a finitely generated $\mathcal{Z}(R)$-module.
\end{itemize}
\end{lemma}
In other words, if $R$ is a semiprime PI right Noetherian ring, then it is 
a finitely generated $\mathcal{Z}(R)$-module iff $\mathcal{Z}(R)$ is Noetherian too.

\medbreak
To proceed with the center of $U(\mathfrak{gl}(1,1))$, we  recall the generalization of the Ha\-rish-Chandra theorem on centers of enveloping algebras of  
semisimple Lie algebras to Lie superalgebras, see \cite[Chapter 13]{Musson}.
Given a Lie superalgebra $\mathfrak{g}$, set 
\begin{align*}
\mathcal{Z}(\mathfrak{g}) \coloneqq \mathcal{Z}(U(\mathfrak{g})).
\end{align*}
\medbreak
A classical simple Lie superalgebra $\mathfrak{g}$ is called \emph{basic} if it admits an even nondegenerate $\mathfrak{g}$-invariant bilinear form, which is necessarily supersymmetric \cite[1.2.4]{Musson}. 
Albeit $\mathfrak{gl}(m,n)$ is not simple,  it bears a nondegenerate invariant supersymmetric even bilinear form, using the supertrace \cite[Exercise 2.7.1]{Musson}.

\medbreak
Let $\mathfrak{g}$ be either a  basic classical simple Lie superalgebra or $\mathfrak{gl}(m,n)$ where $m,n \in \mathbb{N}$.
Let $\mathfrak{h}$ be a Cartan subalgebra of $\mathfrak{g}$;
e.g., the space of diagonal matrices when
$\mathfrak{g} = \mathfrak{gl}(m,n)$. 
Let  $S(\mathfrak{h})^W$ be subalgebra of the symmetric algebra 
$S(\mathfrak{h})$ of $W$-invariant functions 
on $\mathfrak{h}^{\ast}$ where $W$ is the Weyl group.

\medbreak
Now let $I(\mathfrak{h})$  be the subalgebra of $S(\mathfrak{h})^W$
consisting of those $\phi$ such that if $\alpha$ is an
isotropic root and $\lambda \in \mathbb{H}_{\alpha}$, where $\mathbb{H}_{\alpha}$  the hyperplane orthogonal to $\alpha$, then
\begin{align*}\phi(\lambda)  &=  \phi(\lambda + \ell\alpha), &
\text{ for all } \ell \in \ku. 
\end{align*}

Here is the super version of the theorem of Harish-Chandra.

\begin{theorem}\label{centerthm} Assume that $\mathfrak{g}$ is either a basic classical simple Lie superalgebra of type different from $A$ or that $\mathfrak{g} = \mathfrak{gl}(m,n)$ for some positive integers $m$ and $n$.
\begin{enumerate}[leftmargin=*,label=\rm{(\roman*)}]
\item \cite{Kac84}
There is an algebra monomorphism $\psi$ from  $\mathcal{Z}(\mathfrak{g}) $ into $S(\mathfrak{h})^W$.

\item \cites{Kac84,Gorelik04,Sergeev}
The image of the map $\psi$ is $I(\mathfrak{h})$.
\end{enumerate}
\end{theorem}

A detailed exposition of the proof is given in \cite[\S\,13.2]{Musson}. 

\begin{corollary}\label{cor:center-Noetherian}
\cite[Theorem 2.8]{Musson97}, \cite[13.2.11]{Musson}.
For $\mathfrak{g}$ as in Theorem \ref{centerthm}, $\mathcal{Z}(\mathfrak{g}) $ is Noetherian 
if and only if $\mathfrak{g} = \mathfrak{osp}(1,2n)$ for  $n \in \mathbb N$.
\end{corollary}

Combined with Lemma \ref{lema:Noetherian center}, this last result gives:

\begin{corollary}\label{cor:not-fingen-center} $U(\mathfrak{gl}(1,1))$ 
is not finitely generated as a module over its center.
\end{corollary}

\pf The Noetherian algebra $U(\mathfrak{gl}(1,1))$ is prime and PI
by Theorems \ref{primethm} and  \ref{PIthm}. Since 
$\mathcal{Z}(\mathfrak{gl}(1,1))$ is not Noetherian by Corollary \ref{cor:center-Noetherian},
Lemma \ref{lema:Noetherian center} applies.
\epf

\begin{remark}
It is well known that 
$\mathcal{Z}(U(\mathfrak{gl}(1,1)))$ is isomorphic to the subalgebra of 
the polynomial algebra $\ku[x,y]$ of polynomials constant on the line $x=y$, 
which is generated by $x^n - y^n$ for all positive integers $n$.
See \cite[0.6.1]{Sergeev} taking into account \cite[Cor. 13.3.8]{Musson}.
\end{remark}

\medbreak
\subsection{Center of $H(\mathfrak{g})$}\label{centerH} 
Let $\mathfrak{g}$ be a  Lie superalgebra, not necessarily finite dimensional. 
Note that $H(\mathfrak{g})$ is $C_2$-graded if we set the degree of $t$ to be 0,
that is $H(\mathfrak{g}) = H(\mathfrak{g})_0 \oplus H(\mathfrak{g})_1$
where 
\begin{align*}
H(\mathfrak{g})_0 &=  U(\mathfrak{g})_0 \# \ku[t] , & 
H(\mathfrak{g})_1  &= U(\mathfrak{g})_1\# \ku[t].
\end{align*}
 Hence the center $\mathcal{Z}(H(\mathfrak{g}))$ of $H(\mathfrak{g})$ is also $C_2$-graded,
\begin{align*}
\mathcal{Z}(H(\mathfrak{g}))  &=  \mathcal{Z}(H(\mathfrak{g}))_0 \oplus \mathcal{Z}(H(\mathfrak{g}))_1, 
\end{align*}
where
\begin{align*}
\mathcal{Z}(H(\mathfrak{g}))_0 &= \mathcal{Z}(H(\mathfrak{g})) \cap
U(\mathfrak{g})_0 \# \ku[t], &
\mathcal{Z}(H(\mathfrak{g}))_1 &= \mathcal{Z}(H(\mathfrak{g})) \cap
U(\mathfrak{g})_1 \# \ku[t].
\end{align*}

To describe the center of $H(\mathfrak{g})$ we need the following definitions. 
Let $M$ be a graded $U(\mathfrak{g})$-bimodule.
Recall that the adjoint action of $\mathfrak{g}$ on  $M$ is given by
\begin{align*}
\operatorname{ad}(u)(m)  &\coloneqq   um - (-1)^{\vert u \vert \vert m\vert}mu,
\end{align*}
for homogeneous elements $u \in \mathfrak{g}$ and $m \in M$. The submodule of $\mathfrak g$-invariants (with respect to the adjoint action) is
\begin{align*}
M^{\epsilon} &\coloneqq \{m \in M: \operatorname{ad}(u)(m) = 0
\text{ for all } u \in \mathfrak g\}.
\end{align*}
Let $M \mapsto \Pi(M)$ be the parity functor, that interchanges even and odd components.
The \emph{twisted adjoint action} $\operatorname{ad}'$ 
of $\mathfrak{g}$ on $M$ is the adjoint action on
$\Pi(M)$ \cites{Arnaudon-et-al,Gorelik}, i.e., 
\begin{align*}
\operatorname{ad}'(u)(m)  \coloneqq   um - (-1)^{\vert u \vert(\vert m\vert +1)}mu,
\end{align*}
for homogeneous $u\in \mathfrak{g}$ and $m \in M$.
Consider now $U(\mathfrak{g})$ as a bimodule over itself.

\begin{definition}\label{ghost} \cite[\S\,2.1]{Gorelik}
The \emph{anticenter} $A(\mathfrak{g}) \coloneqq \Pi(U(\mathfrak{g}))^{\varepsilon}$ is the 
submodule of $\mathfrak g$-invariants of $U(\mathfrak{g})$ with respect to 
the twisted adjoint action $\mathrm{ad'}$,
that is
\begin{align*}
A(\mathfrak{g})  = \{x \in U(\mathfrak{g}): \operatorname{ad}'(u)(x) = 0
\text{ for all } u \in \mathfrak g\}.
\end{align*}
\end{definition}
Clearly, $A(\mathfrak{g})$ is a graded submodule of $U(\mathfrak{g})$
with respect to $\mathrm{ad'}$.

\medbreak
The following facts are recorded in \cite{Gorelik}.
Notice that the meaning of $\mathcal{Z}(\mathfrak{g})$ in \cite{Gorelik}
is different  than here:
it is the super center, i.e., $U(\mathfrak{g})^{\varepsilon}$ in the notation
above, not the center of the associative algebra $U(\mathfrak{g})$ 
as in the present article. See \cite[\S 3.5]{Gorelik}.

\begin{lemma} \label{lema:properties-Ag} 
\begin{enumerate}[leftmargin=*,label=\rm{(\roman*)}]
\item\label{item:properties-A(g)-1} 
$A(\mathfrak{g})_1 = \mathcal{Z}(\mathfrak{g})_1$.

\medbreak
\item\label{item:properties-A(g)-1.5}  $x \in U(\mathfrak{g})_0$ belongs to $A(\mathfrak{g})_0$ if and only if
\begin{align} \label{eq:A(g)0}
x u_0 &= u_0x && \text{and}& x u_1 &= -u_1x, &
\text{ for all }  u_0   \in \mathfrak g_0, &\quad u_1   \in \mathfrak g_1.
\end{align}

\medbreak
\item\label{item:properties-A(g)-2} $A(\mathfrak{g})$ is a $\mathcal{Z}(\mathfrak{g})_0$-submodule of $U(\mathfrak{g})$.

\medbreak
\item\label{item:properties-A(g)-2.5} If $x,y \in A(\mathfrak{g})_0$,  then $xy \in \mathcal{Z}(\mathfrak{g})_0$.

\medbreak
\item\label{item:properties-A(g)-3}  
If $\operatorname{dim} \mathfrak{g}_1$ is even, then $A(\mathfrak{g}) = A(\mathfrak{g})_0$ and $\mathcal{Z}(\mathfrak{g})  =  \mathcal{Z}(\mathfrak{g})_0$. 
\end{enumerate}
\end{lemma}

\pf \ref{item:properties-A(g)-1} to 
\ref{item:properties-A(g)-2.5} are straightforward.
The first equality in \ref{item:properties-A(g)-3} 
is proved in \cite[Cor. 3.1.3]{Gorelik}; this says that 
$0 = A(\mathfrak{g})_1 = \mathcal{Z}(\mathfrak{g})_1$ 
by\ref{item:properties-A(g)-1}, and so $\mathcal{Z}(\mathfrak{g}) = 
\mathcal{Z}(\mathfrak{g})_0$.
\epf

\medbreak
Now we move on to consider the bosonization $H(\mathfrak{g})$ of $U(\mathfrak{g})$.

\begin{lemma} \label{lema:center-Z(H)}
We have $\mathcal{Z}(H(\mathfrak{g}))_1 =0$, and hence 
\begin{align}
\mathcal{Z}(H(\mathfrak{g})) 
=  \mathcal{Z}(\mathfrak{g})_{0} \oplus  A(\mathfrak{g})_0 \, t.
\end{align}
\end{lemma}

\pf
Let $x \in \mathcal{Z}(H(\mathfrak{g}))_1$. Then $tx = xt$ by centrality
but $tx = -xt$, so $x = 0$. 
Let  $x = a + bt\in H(\mathfrak{g})_0$, 
where $a, b \in U(\mathfrak{g})_0$. Given $u_0 \in \mathfrak g_0$ and 
$u_1 \in \mathfrak g_1$, we have
\begin{align*}
u_0x &= xu_0 &&\Leftrightarrow&  u_0(a + bt) &= (a + bt)u_0 
&&\Leftrightarrow&  u_0a  = au_0 \text{ and } u_0b &= bu_0 ;
\\
u_1x &= xu_1 &&\Leftrightarrow&  u_1(a + bt) &= (a + bt)u_1 
&&\Leftrightarrow&  u_1a  = au_1 \text{ and } u_1b &= -bu_1.
\end{align*}
Since $x$ commutes with $t$, we see that
$x\in \mathcal{Z}(H(\mathfrak{g}))$ iff
$a\in  \mathcal{Z}(\mathfrak{g})_{0}$ and $b \in A(\mathfrak{g})_0$.
\epf

\medbreak

We can now deduce that the centers of the bosonizations of many enveloping algebras of Lie superalgebras are not Noetherian:
\begin{proposition}\label{cor:center-H(g)-Noetherian}
Let $\mathfrak{g}$ be as in Theorem \ref{centerthm}, 
but $\mathfrak{g} \not\simeq \mathfrak{osp}(1,2n)$ for any $n \in \mathbb N$.
Assume that $\dim \mathfrak{g}_1$ is even. 
Then $\mathcal{Z} (H(\mathfrak{g})) $ is not Noetherian.
\end{proposition}

\pf
By Lemmas \ref{lema:properties-Ag} \ref{item:properties-A(g)-3} and \ref{lema:center-Z(H)}, as we assume 
that $\dim \mathfrak{g}_1$ is even, we have
\begin{align}\label{caught}
\mathcal{Z}(H(\mathfrak{g})) 
=  \mathcal{Z}(\mathfrak{g})_{0} \oplus  A(\mathfrak{g})_0 \, t
=  \mathcal{Z}(\mathfrak{g}) \oplus  A(\mathfrak{g}) \, t,
\end{align}
which is a decomposition of $\mathcal{Z}(\mathfrak{g})$-modules
by Lemma \ref{lema:properties-Ag} \ref{item:properties-A(g)-2}.
Also $A(\mathfrak{g}) \, tA(\mathfrak{g}) \, t \subset \mathcal{Z}(\mathfrak{g})$
by Lemma \ref{lema:properties-Ag}   \ref{item:properties-A(g)-2.5}.
Since $\mathcal{Z}(\mathfrak{g})$ is not Noetherian by Corollary \ref{cor:center-Noetherian}, 
we conclude that $\mathcal{Z}(H(\mathfrak{g}))$ 
is not Noetherian 
by Lemma \ref{lema:R-noeth} below. 
\epf

For completeness, we give a proof of the following result.

\begin{lemma}\label{lema:R-noeth}
Let $S$ be a subring of a Noetherian commutative ring $R$  that admits
a $S$-sub\-module $T$  such that $R = S \oplus T$ and $T \cdot T \subset S$. 
Then $S$ is  Noetherian.
\end{lemma}

\pf  If $J$ is an ideal in $S$, then $I = J \oplus JT$ is an ideal of $R$:
\begin{align*}
R I = (S+T)(J + JT ) = SJ + SJT + JT + TJT \subset J + JT.
\end{align*}
Clearly $I \cap S = J$. 
Given an ascending chain $\mathfrak C$
of ideals in $S$, we get in this way an ascending chain 
of ideals in $R$ that stabilizes because $R$ is Noetherian. Intersecting 
with $S$, we see that $\mathfrak{C}$ stabilizes, so $S$ is Noetherian.
\epf

We can now prove  Theorem \ref{thm:main}, namely that  
$H(\mathfrak{gl}(1,1))$ is not finitely generated as a module over its center.

\medbreak
 \textsc{Proof of Theorem \ref{thm:main}.}
The Noetherian algebra $H(\mathfrak{gl}(1,1))$ is semiprime by Lemma 
\ref{lema:GL-gap} and PI by  Corollary \ref{PIcor}. Since the center 
$\mathcal{Z}(H(\mathfrak{gl}(1,1)))$ is not Noetherian by Proposition \ref{cor:center-H(g)-Noetherian},
Lemma \ref{lema:Noetherian center} applies. 
\qed

\medbreak
In fact, an explicit description of $\mathcal{Z}(H(\mathfrak{gl}(1,1)))$
follows from Proposition \ref{cor:center-H(g)-Noetherian}:

\begin{theorem}\label{thm3} \begin{enumerate}
\item[{\rm(i)}] $\mathcal{Z}(H(\mathfrak{gl}(1,1))) \, = \mathcal{Z}(\mathfrak{gl}(1,1)) \oplus A(\mathfrak{gl}(1,1))  \, t$. 
\item[{\rm(ii)}] Let $\tau$ be the algebra automorphism of $\ku [x,y]$ defined by $\tau(x) = x$ and $\tau (y) = y-1.$ Then
$$ A(\mathfrak{gl}(1,1)) \; = \; \{ x\omega - (\omega + \tau (\omega))uv : \omega \in \ku [x,y]\}.$$
\end{enumerate} 
\end{theorem}

\pf (i) Since $\operatorname{dim} \mathfrak{gl}(1,1)$ is even, this follows from \eqref{caught}.

\smallbreak
(ii) Let $\alpha \in A(\mathfrak{gl}(1,1)) = A(\mathfrak{gl}(1,1))_0 \subseteq U(\mathfrak{gl}(1,1))_0 = \ku [x,y] \oplus \ku [x,y]uv$, so that 
$\alpha = r(x,y) + s(x,y)uv$ for unique elements $r = r(x,y),\, 
s = s(x,y) \in \ku [x,y]$. 
Since $U(\mathfrak{gl}(1,1))_0$ is commutative we see from \eqref{eq:A(g)0} that 
\begin{equation}\label{condition} \alpha \in A(\mathfrak{gl}(1,1)) \Longleftrightarrow \alpha u = -u\alpha \textit{ and } \alpha v = - v \alpha.
\end{equation}
By \eqref{eq:def-rels-gl1,1}, we have that $yu= u(y + 1) $, $yv = v(y - 1)$, hence 
\begin{align*}
f(x,y) u &= uf(x, y+1), & f(x,y) v &= vf(x, y-1), & \text{for any  } f(x,y) &\in \ku [x,y].
\end{align*}
Therefore, since $uvu = ux$ and $vuv = vx$, we see that
\begin{align*}
\alpha u &=  \left(r(x,y) + s(x,y)uv\right)u = u \left(r(x,y+1) + xs(x,y+1)\right);
\\
u\alpha  &=  u\left(r(x,y) + s(x,y)uv\right) = ur(x,y);
\\
 \alpha v &=  \left(r(x,y) + s(x,y)uv\right)v = v r(x,y - 1);
 \\
v\alpha  &=  v\left(r(x,y) + xs(x,y)\right).
\end{align*}
Hence  both equalities in the right hand side of \ref{condition} translate to
the same condition $r(x,y) + xs(x,y)  = -r(x,y-1)$, that is
\begin{equation}\label{key} \alpha \in A(\mathfrak{gl}(1,1)) \Longleftrightarrow r + \tau (r) = - xs.
\end{equation}
Writing $r = \sum_{0 \leq j \leq m}r_j x^j$ for unique  $r_j \in \ku [y]$ and substituting in \eqref{key}, we find that 
\begin{equation}\label{last} r_0 + \tau (r_0) \; = \; 0;
\end{equation}
considering the highest degree terms in $y$ in \eqref{last}, we see that $r_0 = 0$. 
Equivalently, $r = x\omega$ for $\omega \in \ku [x,y]$. 
Substituting in 
\eqref{key} we get that $s = - \omega - \tau(\omega)$ for $\omega \in \ku [x,y]$.  
Since all the steps can be reversed, the proof is complete.
\epf

\section{On Noetherian PI Hopf algebras finite over their centers}\label{positive} 
As is known, examples of Noetherian prime affine PI algebras  that are \emph{not} finite over their center are scarce. 
Here is an example: it is shown in \cite[\S\,13.10.2]{Mcconell-Robson}  
that a subalgebra $S$ of the algebra $R$ introduced in \cite[\S\,5.3.7 (iii)]{Mcconell-Robson} is affine  Noetherian prime  PI, but $\mathcal{Z}(S)$ is not Noetherian. 
Thus $S$ is not finite over $\mathcal{Z}(S)$ by Lemma \ref{lema:Noetherian center}.

\medbreak
On the positive side, here is a criterion, obtained by combining a couple of deep results,
that apparently has not been recorded elsewhere.

\begin{theorem}\label{gldimthm} Let $H$ be an affine Noetherian Hopf algebra satisfying a polynomial identity. Suppose that the global homological dimension of $H$ is finite. Then $H$ is a finite direct sum of prime rings, and is a finite module over its center.
\end{theorem}

\begin{proof} By \cite[Theorems 0.1, 0.2]{Wu-Zhang}, $H$ is Auslander-Goren\-stein 
and AS-Goren\-stein. Thus, if $\operatorname{gldim} H < \infty$, it is Auslander-regular 
and AS-regular. By \cite[1.4]{Stafford-Zhang}, $H$ is a finite direct sum of 
prime rings and is integral over its center. One easily sees that an affine ring $R$
which is integral over $\mathcal{Z}(R)$ is a finite module over  $\mathcal{Z}(R)$.
\end{proof}

\begin{remark}
The results quoted from \cite{Wu-Zhang} are valid over all fields $\ku$, and the 
hypothesis that $H$ is affine can be weakened to ``all simple $H$-modules are finite 
dimensional over $\ku $''. 
\end{remark}

\begin{remark} Note that it follows immediately from Theorems \ref{thm:main} and \ref{gldimthm} that  
\begin{align}\label{eq:gldim}
\operatorname{gldim}  H(\mathfrak{gl}(1,1)) =  \infty.
\end{align}
\noindent But in fact it is easy to see this directly: for consider the Hopf subalgebra $S \coloneqq  \ku \langle  u,t\rangle$ of $H(\mathfrak{gl}(1,1))$, which is isomorphic to the  Sweedler Hopf algebra. Clearly $\operatorname{gldim} S = \infty$ and $H(\mathfrak{gl}(1,1))$ is free over $S$ by the PBW theorem. By a standard argument, see e.g. \cite[Lemma 4.25]{A-Natale-Torrecillas}, \eqref{eq:gldim} follows. Notice that $S$ could be replaced in the above argument by the subalgebra $R \coloneqq  \ku \langle  u \rangle$.
\end{remark}

\begin{remark} Returning to an issue raised in \S 1, namely whether an affine Noetherian Hopf algebra $H$ which is a domain and satisfies a PI has to be finite over $\mathcal{Z}(H)$, we note that this cannot simply be proved by showing that $\operatorname{gldim} H$ is finite and applying Theorem \ref{gldimthm}. For the  affine Hopf domains $B(n, \{p_i\}_{i=1}^s,q, \{\alpha_i\}_{i=1}^s)$ constructed in \cite{Wang-Zhang-Zhuang} are Noetherian, 
finite modules over their centers but have finite global dimension only for some very special values of the parameters. 
\end{remark}

\begin{remark} Taking further a comment made in the opening paragraph of \S 1, 
it would be interesting to look for conditions ensuring that a prime Noetherian Hopf algebra is a maximal order. For example, filtered Noetherian domains whose associated graded rings are maximal orders inherit the same property \cite[Theorem 5.1.6]{Mcconell-Robson}. It may also be more than a coincidence that, in the characterisation \cite[Theorem F]{Brown-maxord} of those group algebras of polycyclic-by-finite groups which are maximal orders, the prime 2 is key, as is the case in the Lie superalgebra setting.

\end{remark}

\subsection*{Acknowledgements} We thank Ken Goodearl, Ed Letzter, Ian Musson and Vera Sergano\-va for providing useful information.
N.A.  thanks Efim Zelmanov and Slava Futorny for their warm hospitality durin his visit 
to the SUSTech.

\end{document}